\documentclass[amstex,a4]{article}
 \usepackage{amsfonts}
 \usepackage{amsmath}
 \usepackage{amsthm}
 \usepackage{bm}
 \usepackage{amssymb}

 \newtheorem{theorem}{Theorem}[section]

\title{\textbf{Uniqueness for inverse problems of determining orders of 
multi-term time-fractional derivatives of diffusion equation}$^*$}

\author{\textbf{Zhiyuan Li$^\dag$}, \textbf{Masahiro Yamamoto$^\dag$}}

\date{}

\begin{document}
 \maketitle
 
\begin{center}
$^\dag$ Graduate School of Mathematical Sciences, the University of Tokyo, 
3-8-1 Komaba, Meguro-ku, Tokyo 153-8914, Japan.\\
E-mail: zyli@ms.u-tokyo.ac.jp, myama@ms.u-tokyo.ac.jp\\
\end{center}
 
\renewcommand{\thefootnote}{\fnsymbol{footnote}}

\footnotetext{
$^*$ Project supported by the Leading Graduate Course for Frontiers of 
Mathematical Sciences and Physics (The University of Tokyo).
}

 
\begin{abstract}
This article proves the uniqueness for two kinds of inverse problems
of identifying fractional orders in diffusion equations with multiple 
time-fractional derivatives by pointwise observation. 
By means of eigenfunction expansion and Laplace transform, 
we reduce the uniqueness for our inverse problems to the uniqueness 
of expansions of some special function and complete the proof.
\end{abstract}


orders,
  eigenfunction expansion,  multi-time-fractional derivatives


\section{Introduction}
Recently, the fractional diffusion equations have been successfully 
used for modelling physical processes such as anomalous diffusion 
(e.g., Metzler and Klafter \cite{MK}). For better modelling, 
a fractional diffusion equation with multi-term time-fractional 
derivatives has been proposed (e.g., Luchko \cite{Lu09},
\cite{Luchko2} and the references therein), which is given by:
\begin{equation}\label{equ-intro-u}
\sum_{j=1}^{n}q_j \partial_t^{\alpha_j} u(x,t) = -Au(x,t), \quad
x\in \Omega, \thinspace t > 0.
\end{equation}
Here $\Omega$ is a bounded domain in $\mathbb{R}^d$ with smooth boundary
$\partial\Omega$, $q_j>0$, $\alpha_j \in (0,1)$, $j=1,2,..., n$ are 
constants, and $-A$ is an elliptic operator, and we define the 
Caputo fractional derivative with respect to $t$ by
$$
  \partial_t^{\alpha} g(t) 
= \frac{1}{\Gamma(1-\alpha)} 
  \int_0^t (t-\tau)^{-\alpha}\frac{d}{d\tau} g(\tau) d\tau, \quad
0 < \alpha < 1, 
$$
where $\Gamma$ is the Gamma function (e.g., Podlubny \cite{Podlubny}).
Beckers and Yamamoto \cite{BY}, Li and Yamamoto \cite{LY} and
Li, Liu and Yamamoto \cite{LLY} discuss some properties as well as
the unique existence of solution to an initial value - boundary value 
problem for (\ref{equ-intro-u}). 

When we consider (\ref{equ-intro-u}) as model equation for describing 
e.g., anomalous diffusion in inhomogeneous media, the orders $\alpha_j$
of fractional derivatives should be determined by the inhomogeneity of 
the media, but it is not clear which physical law can correspond the 
inhomogeneity to the orders $\alpha_j$.  Thus one reasonable way for 
estimating $\alpha_j$ is an inverse problem of determining 
$\alpha_1, ..., \alpha_n$ in order to match available data such as 
$u(x_0,t)$, $0 < t < T$ at a monitoring point $x_0 \in \Omega$. 
In this article, we establish the uniqueness of our inverse problems, 
and the uniqueness is the first fundamental theoretical subject. 
As for inverse problems for fractional diffusion equations with a single 
time-fractional derivative, see for Cheng, Nakagawa, Yamamoto and Yamazaki 
\cite{CNYY}, Hatano, Nakagawa, Wang and Yamamoto \cite{HNWY}, for example.

In this article, we consider two kinds of inverse problems in determining
orders $\alpha_j$ and other quantities. In section 2, an initial value is 
assumed to concentrate on $x=0$, i.e., the Dirac function $\delta(x)$.
The uniqueness in determining orders $\alpha_j$ by measured data at one 
endpoint is proved. In section 3, choosing an initial value in $L^2(\Omega)$ 
suitably, we establish the uniqueness.


\section{Inverse problem with $\delta$ initial value}
In this section, we consider the following initial value - boundary value 
problem
\begin{equation} \label{equ-1d-u}
   \left\{
  {\begin{array}{*{20}c}
  {\sum_{j=1}^{n}q_j \partial_t^{\alpha_j} u(x,t)
   =\partial_x \big(p(x)\partial_xu(x,t)\big)+c(x)u(x,t)},\quad (x,t)\in(0,\ell)\times(0,T), 
\hfill \\
  {u(x,0)=\delta(x),\quad x\in (0,\ell),} \hfill \\
  {\partial_xu(0,t)=\partial_xu(\ell,t)=0,\quad t \in (0,T).} \hfill
 \end{array} } \right.
 \end{equation}
Here $T>0$, $\ell>0$ are fixed and $\delta(x)$ is the Dirac function.

We assume that $c\in C[0,\ell]$, $c\leq0$ on $[0,\ell]$, 
and $p\in C^2[0,\ell]$ is known as a positive function, 
and $0<\alpha_1<\cdots<\alpha_{n-1}<\alpha_n<1$, $q_j>0$, $j=1,\cdots,n$, 
are constants. The initial condition in (\ref{equ-1d-u}) means that we start 
experiments by setting up a density profile concentrating on $x=0$, 
and the boundary data of (\ref{equ-1d-u}) requires no fluxes at both endpoints.

Let $I$ be a non-empty open interval in $(0,T)$. We discuss

{\bf\textit{Inverse problem:}} 
Determine the number $n$ of fractional orders $\alpha_j$, fractional orders 
$\{\alpha_j\}_{j=1}^n$ of the time derivatives, and constant coefficients 
$\{q_j\}_{j=1}^n$ of the fractional derivatives from boundary measurement 
$u(0,t)$, $t\in I$.

First we define an operator $A_p$ in $L^2(0,\ell)$ by 
$$
\left\{
  \begin{array}{*{20}c}
  { (A_p\psi)(x)=-\frac{d}{dx}\left( p(x) \frac{d}{dx}\psi(x)\right)-c(x)\psi(x)},
 \quad 0<x<\ell, \hfill \\
  { \mathcal{D}(A_p)
=\left\{ \psi\in H^2(0,\ell); \frac{d\psi}{dx}(0)=\frac{d\psi}{dx}(\ell)=0 
\right\}. }\hfill
 \end{array}\right.
$$
Let $\{\lambda_k,\varphi_k\}_{k=1}^{\infty}$ be an eigensystem of the elliptic 
operator $A_{p}$.  That is, 
$0=\lambda_1<\lambda_2<\cdots,\ \lim_{k\rightarrow\infty}\lambda_k=\infty$,
$A\varphi_k = \lambda_k\varphi_k$, and $\{\varphi_k\}\subset H^2(0,\ell)$ 
forms an orthogonal basis of $L^2(\Omega)$ with $\varphi_k(0) = 1$.

Now we define the multinomial Mittag-Leffler function (see Luchko and Gorenflo 
\cite{LG99}):
$$
E_{(\theta_1,\cdots,\theta_n),\theta_0}(z_1,\cdots,z_n)
:=\sum_{k=0}^\infty\sum_{k_1+\cdots+k_n=k}{(k;k_1,\cdots,k_n)
\prod_{j=1}^nz_j^{k_j}}{\Gamma\left(\theta_0+\sum_{j=1}^n\theta_jk_j
\right)},
$$
where $0<\theta_0<2$, $0<\theta_j<1$, $z_j\in\mathbb{C}$ ($j=1,\cdots,n$), 
and $(k;k_1,\cdots,k_n)$ denotes the multinomial coefficient
$$
(k;k_1,\cdots,k_n):={k!}{k_1!\cdots k_n!}\quad\mbox{with }k=\sum_{j=1}^nk_j
$$
(see also \cite{LLY}). For later use, we adopt the abbreviation
$$
  E_{\bm{q},\bm{\alpha}',1+\alpha_n}^{(j)}(t)
:=E_{(\alpha_n,\alpha_n-\alpha_1,\cdots,\alpha_n-\alpha_{n-1}),1+\alpha_n}
  \left(-\frac{\lambda_j}{q_n}t^{\alpha_n},-\frac{q_1}{q_n}t^{\alpha_n-\alpha_1},
        \cdots,-\frac{q_{n-1}}{q_n} t^{\alpha_n-\alpha_{n-1}}\right),\quad t>0.
$$
Similarly to \cite{CNYY}, by means of the eigensystem $\{\lambda_k, \varphi_k\}
^{\infty}_{k=1}$, we can define the weak solution to (\ref{equ-1d-u}), but we here 
omit the details.
 
\begin{theorem}[Uniqueness] \label{thm-uniqu-delta}
Let us assume $p\in C^2[0,\ell]$, $p>0$ on $[0,\ell]$ and $c\in C[0,\ell]$, $c\leq 0$ on $[0,\ell]$. Let $u$ be the weak 
solution to $(\ref{equ-1d-u})$, and let $v$ be the weak solution to $(\ref{equ-1d-v})$ 
with the same initial and boundary conditions as $(\ref{equ-1d-u})$:
\begin{equation} \label{equ-1d-v}
   \left\{
  {\begin{array}{*{20}c}
  {\sum_{j=1}^{m}r_j \partial_t^{\beta_j} v(x,t)
   =\partial_x \big(p(x)\partial_xv(x,t)\big)+c(x)v(x,t)},\ (x,t)\in(0,\ell)\times(0,T), 
\hfill \\
  {v(x,0)=\delta(x),\ x\in (0,\ell),} \hfill \\
  {\partial_xv(0,t)=\partial_xv(\ell,t)=0,\quad t \in (0,T),} \hfill
 \end{array} } \right.
 \end{equation}
where $0 < \beta_1 < \cdots < \beta_m < 1$ and $r_j>0$, $j=1,\cdots,m$ are 
constants. Then $u(0,t)=v(0,t)$ in $I$ implies $m=n$, $\alpha_j=\beta_j$ 
and $q_j=r_j$, $j=1,\cdots,n$.
\end{theorem}
\begin{proof}
By an argument similar to the proof in \cite{LLY}, we can give the solutions 
to (\ref{equ-1d-u}) and (\ref{equ-1d-v}) by
\begin{align}
u(\cdot,t)
=\sum_{j=1}^{\infty} \sigma_j
   (1-\lambda_jt^{\alpha_n} E_{\bm{q},\bm{\alpha}',1+\alpha_n}^{(j)}(t))
\langle\delta,\varphi_j\rangle\varphi_j
=\sum_{j=1}^{\infty} \sigma_j
   (1-\lambda_jt^{\alpha_n} E_{\bm{q},\bm{\alpha}',1+\alpha_n}^{(j)}(t))
\varphi_j\label{sol-1d-u},\\
v(\cdot,t)
=\sum_{j=1}^{\infty} \sigma_j
   (1-\lambda_jt^{\beta_m} E_{\bm{r},\bm{\beta}',1+\beta_m}^{(j)}(t))
\langle\delta,\varphi_j\rangle\varphi_j
=\sum_{j=1}^{\infty} \sigma_j
(1-\lambda_jt^{\beta_m} E_{\bm{r},\bm{\beta}',1+\beta_m}^{(j)}(t))
\varphi_j\label{sol-1d-v},
\end{align}
where $\sigma_j=\|\varphi_j\|_{L^2(\Omega)}^{-2}$, and there exists a constant 
$c_0>0$ such that $\sigma_j=c_0+o(1)$, $j\rightarrow\infty$. Let $t_0>0$ and 
$M>0$ be arbitrarily fixed. By the Sobolev embedding theorem and Lemma 2.2 in 
\cite{LLY}, noting that $\lambda_j \sim j^2$, 
that is, $C_0j^2 \leq \lambda_j \leq C_1j^2$ (see, e.g., \cite{CH}), 
we find
\begin{align*}
&\sum_{j=1}^{\infty} \sigma_j \|(1-\lambda_jt^{\alpha_n}
E_{\bm{q},\bm{\alpha}',1+\alpha_n}^{(j)}(t))\varphi_j\|_{C[0,\ell]}
\leq C \sum_{j=1}^{\infty} 
\sigma_j \sum_{i=1}^{n-1} \frac{t^{\alpha_n-\alpha_i}}{1+\lambda_jt^{\alpha_n}}
           \|(A_p+M)^{\frac{1}{4}+\varepsilon}\varphi_j\|_{L^2(0,\ell)}\\
\leq& C\sum_{j=1}^{\infty} |\sigma_j| 
\sum_{i=1}^{n-1} \frac{t^{\alpha_n-\alpha_i}}{1+\lambda_jt^{\alpha_n}}
(\lambda_j+M)^{\frac{1}{4}+\varepsilon}<\infty,\quad t_0\leq t\leq T.
\end{align*}
Therefore we see that the series on the right-hand side of (\ref{sol-1d-u}) 
and (\ref{sol-1d-v}) are convergent uniformly in $x\in [0,\ell]$ and $t\in[t_0,T]$.
Moreover, since the solutions $u$ and $v$ can be analytically extended to
$t>0$ in view of the analyticity of the multinomial Mittag-Leffler function
(see also \cite{LIY}, \cite{LY}), we see that $u(0,t)=v(0,t)$, $t\in I$, 
implies
\begin{align}\label{u(0,t)=v(0,t)}
\sum_{j=1}^{\infty} \sigma_j
\left(1-\lambda_jt^{\alpha_n} E_{\bm{q},\bm{\alpha}',1+\alpha_n}^{(j)}(t)
\right)
=\sum_{j=1}^{\infty} \sigma_j
\left(1-\lambda_jt^{\beta_m} E_{\bm{r},\bm{\beta}',1+\beta_m}^{(j)}(t)
\right),\quad t>0.
\end{align}
Taking the Laplace transforms on both sides of (\ref{u(0,t)=v(0,t)}), 
we find
\begin{align}\label{Lu=Lv}
\sum_{k=2}^{\infty} \sigma_k \frac{\sum_{j=1}^{n} q_j s^{\alpha_j}}
{\sum_{j=1}^{n} q_j s^{\alpha_j}+\lambda_k}
=\sum_{k=2}^{\infty} \sigma_k \frac{\sum_{j=1}^{m} r_j s^{\beta_j}}
{\sum_{j=1}^{m} r_j s^{\beta_j}+\lambda_k},
\quad \mbox{$s\in\mathbb{C}$ with $|s|$ small enough}.
\end{align}
We see that the series on both sides of (\ref{Lu=Lv}) are convergent uniformly
in $x\in[0,\ell]$ and $s \neq0, \in \mathbb{C}$, $|s|$ small enough. 

Now we prove that $m=n$, and for any $j\in \{1,\cdots,n\}$, $\alpha_j = \beta_j$ 
and $q_j = r_j$.  The proof is done by contradiction.
We assume that $m\neq n$ or else $m=n$ but there exists $j\in \{1,\cdots,n\}$
such that $\alpha_j \neq \beta_j$ or $q_j \neq r_j$.
We set $w_1(s):=\sum_{j=1}^{n} q_j s^{\alpha_j}$ and 
$w_2(s):=\sum_{j=1}^{m} r_j s^{\beta_j}$.
Then we see that there exists a small real number $s_0>0$ such that 
$w_1(s_0)\neq w_2(s_0)$.  
By $q_j, r_j > 0$, we have $w_1(s_0), w_2(s_0) > 0$.  
Therefore $w_1(s_0)+\lambda_k > 0$ and $w_2(s_0)+\lambda_k > 0$ for all 
$k=2,3,\cdots$ in view of $\lambda_k>0$ for $k=2,3,\cdots$.

Since $w_1(s_0)\ne w_2(s_0)$, we can assume that $w_1(s_0)> w_2(s_0)$. 
Then 
$$
\frac{w_1(s_0)}{w_1(s_0)+\lambda_k}>\frac{w_2(s_0)}{w_2(s_0)+\lambda_k},\ 
k=2,3,\cdots
$$
by $w_1(s_0)+\lambda_k > 0$ and $w_2(s_0)+\lambda_k > 0$ for all $k=2,3,\cdots$.
By the above inequality and $\sigma_k>0$, we deduce that
$$
\sum_{k=2}^{\infty}  \sigma_k \frac{\sum_{j=1}^{n} q_j s_0^{\alpha_j}}
{\sum_{j=1}^{n} q_j s_0^{\alpha_j}+\lambda_k}
>\sum_{k=2}^{\infty} \sigma_k \frac{\sum_{j=1}^{m} r_j s_0^{\beta_j}}
{\sum_{j=1}^{m} r_j s_0^{\beta_j}+\lambda_k},
$$
which is a contradiction.  Hence we have $m=n$, 
and for any $j\in \{1,\cdots,n\}$, $\alpha_j = \beta_j$ and $q_j = r_j$.
Thus the proof of the theorem is completed.
\end{proof}


\section{Inverse problem  with $L^2(\Omega)$ initial value}
In this section, we consider a bounded domain $\Omega \subset \mathbb{R}^d$ 
with smooth boundary $\partial\Omega$.  
Let $T>0$ be fixed arbitrarily. 
Consider the following initial value - boundary value problem
\begin{equation} \label{equ-d-u}
   \left\{
  \begin{array}{*{20}c}
  \sum_{i=1}^{n}q_j\partial_t^{\alpha_i} u(x,t)
   =\sum_{i,j=1}^d {\partial_j(a_{ij}\partial_iu(x,t))} 
+ c(x)u(x,t),\ (x,t)\in\Omega\times(0,T), \hfill \\
  {u(x,0)=a(x),\ x\in \Omega,} \hfill \\
  {u(x,t)=0,\ (x,t) \in\partial\Omega\times (0,T),} \hfill
 \end{array} \right.
 \end{equation}
where $\alpha_j$ and $q_j>0$, $j=1,\cdots,n$, are constants such that 
\begin{align}\label{condi-alpha}
0<\alpha_1<\cdots<\alpha_n<1,
\end{align}
$a_{ij}=a_{ji}$, $1\leq i,j\leq d$, and $c\leq0$ in $\overline{\Omega}$.
Moreover, it is assumed that $a_{ij}\in C^1(\overline{\Omega})$ and
$c\in C(\overline{\Omega})$, and there exists a constant $\mu>0$ such that
$$
\mu\sum_{i=1}^d \xi_i^2\leq \sum_{i,j=1}^d a_{ij}(x)\xi_i\xi_j,\quad 
\forall x\in\overline{\Omega},
\forall (\xi_1,\cdots,\xi_d)\in\mathbb{R}^d.
$$
Now we define operator $A$ in $H^2(\Omega)\cap H_0^1(\Omega)$ as follows:
\begin{equation*}
 (A\psi)(x)=-\sum_{i,j=1}^d {\partial_j(a_{ij}(x)\partial_i\psi(x))}
 - c(x)\psi(x),\quad x\in\Omega,\quad
 \psi\in H^2(\Omega)\cap H_0^1(\Omega). 
 \end{equation*}
Let $\{\lambda_k, \phi_k\}_{k=1}^{\infty}$ be an eigensystem of 
the elliptic operator $A$:
$0<\lambda_1<\lambda_2<\cdots,\ \lim_{k\rightarrow\infty}\lambda_k=\infty$,
and $A\phi_k=\lambda_k\phi_k$, $\{\phi_k\}_{k=1}^{\infty}\subset 
H^2(\Omega)\cap H_0^1(\Omega)$ forms an orthogonal basis of $L^2(\Omega)$.

Henceforth $(\cdot, \cdot)$ denotes the scalar product in $L^2(\Omega)$.
Moreover we can define a fractional power $A^{\gamma}$ of $A$ with 
$\gamma > 0$ (e.g., Tanabe \cite{Ta}).

We discuss

{\bf\textit{Inverse problem:}} 
Let $x_0 \in \Omega$ be fixed and let $I \subset (0,T)$ be a non-empty 
open interval.
Determine the number $n$ of fractional orders $\alpha_j$, fractional orders 
$\{\alpha_j\}_{j=1}^n$ of the time derivatives, and constant coefficients 
$\{q_j\}_{j=1}^n$ of the fractional derivatives from interior measurement 
$u(x_0,t)$, $t\in I$.

\begin{theorem}[Uniqueness]\label{thm-uniqu-L2}
Assuming that $a\geq0$ in $\Omega$, $a\neq0$ and $a\in D(A^{\gamma})$ with 
$\gamma>\max\{\frac{d}{2}+\delta-1,0\}$, $\delta>0$ can be sufficiently small. 
Let $u$ be the weak solution to $(\ref{equ-d-u})$, and let $v$ be the weak 
solution to $(\ref{equ-d-v})$ with the same initial and boundary conditions 
as $(\ref{equ-d-u})$:
\begin{equation} \label{equ-d-v}
   \left\{
  {\begin{array}{*{20}c}
  {\sum_{j=1}^m r_j \partial_t^{\beta_j} v(x,t)
   =\sum_{i,j=1}^d {\partial_j(a_{ij}\partial_iv(x,t))} + c(x)v(x,t)},
\ (x,t)\in\Omega\times(0,T), \hfill \\
  {v(x,0)=a(x),\ x\in \Omega,} \hfill \\
  {v(x,t)=0,\quad (x,t) \in \partial\Omega\times(0,T),} \hfill
 \end{array} } \right.
 \end{equation}
where $r_i>0$, $i=1,\cdots,m$ are constants, and
\begin{align}\label{condi-beta}
0<\beta_1<\cdots<\beta_m<1.
\end{align}
Then for any fixed $x_0\in\Omega$, $u(x_0,t)=v(x_0,t)$, $t\in I$, 
implies $m=n$, $\alpha_i=\beta_i$, $q_i=r_i$, $i=1,\cdots,n$.
\end{theorem}

\begin{proof}
We know that
\begin{align}
u(\cdot,t)
&=\sum_{j=1}^{\infty}
(1-\lambda_jt^{\alpha_n}E_{\bm{q},\bm{\alpha}',1+\alpha_n}^{(j)}(t))(a,\phi_j)
\phi_j, \label{sol-d-u}\\
v(\cdot,t)
&=\sum_{j=1}^{\infty}
(1-\lambda_jt^{\beta_m} E_{\bm{r},\bm{\beta}' ,1+\beta_m}^{(j)}(t)) (a,\phi_j)
\phi_j \quad \mbox{in $L^2(\Omega)$} \label{sol-d-v} 
\end{align}
for each $t \in [0,T]$ (e.g., Theorem 2.4 in \cite{LLY}). The Sobolev embedding 
inequality yields that 
$
     \|\phi_j\|_{C(\overline{\Omega})}
\le C\|A^{\frac{d}{4}+\varepsilon}\phi_j\|_{L^2(\Omega)}
$ 
with sufficiently small $\varepsilon>0$, and we have 
$C_0j^{\frac{2}{d}} \leq \lambda_j \leq C_1j^{\frac{2}{d}}$ (see, e.g., \cite{CH}). 
Therefore, fixing $t_0 > 0$ arbitrarily, similarly to the proof of 
Theorem \ref{thm-uniqu-delta}, for $t \in [t_0, T]$, we obtain
\begin{eqnarray*}
&&\sum_{j=1}^{\infty}
  |(1-\lambda_jt^{\alpha_n}E_{\bm{q},\bm{\alpha}',1+\alpha_n}^{(j)}(t))|
\|(a,\phi_j)\phi_j\|_{C(\overline{\Omega})}
\leq C\sum_{j=1}^{\infty} 
      \sum_{i=1}^{n-1} \frac{t^{\alpha_n-\alpha_i}} {1+\lambda_jt^{\alpha_n}}
      \|(A^{\gamma}a,\phi_j)A^{-\gamma}\phi_j\|_{C(\overline{\Omega})}\\
\leq&& C \sum_{j=1}^{\infty} |(A^{\gamma}a,\phi_j)|
\sum_{i=1}^{n-1}\frac{t^{\alpha_n-\alpha_i}}{1+\lambda_jt^{\alpha_n}}
\lambda_j^{\frac{d}{4}+\varepsilon-\gamma}
\le C\sum_{j=1}^{\infty} \vert (A^{\gamma}a,\phi_j)\vert
\lambda_j^{\frac{d}{4}+\varepsilon-\gamma-1}\\
\le&& C\left(\sum_{j=1}^{\infty} \vert (A^{\gamma}a,\phi_j)\vert^2
\right)^{\frac{1}{2}}\left( \sum_{j=1}^{\infty} 
\lambda_j^{\frac{d}{2}+2\varepsilon-2\gamma-2}\right)^{\frac{1}{2}}.
\end{eqnarray*}
By $\lambda_j \sim j^{\frac{2}{d}}$ as $j \to\infty$ (e.g., \cite{CH}) 
and $\gamma > \frac{d}{2}-1$, we see that
$
\sum_{j=1}^{\infty}\lambda_j^{\frac{d}{2}+2\varepsilon-2\gamma-2}<\infty
$.  
Hence
\begin{equation}\label{unif-u}
\sum_{j=1}^{\infty}|(1-\lambda_jt^{\alpha_n}
E_{\bm{q}, \bm{\alpha}',1+\alpha_n}^{(j)}(t))|
  \|(a,\phi_j)\phi_j\|_{C(\overline{\Omega})}
< \infty, \quad t_0 \le t \le T.
\end{equation}
Therefore, we see that the series on the right-hand side of (\ref{sol-d-u})
and (\ref{sol-d-v}) are convergent uniformly in $x\in\overline{\Omega}$ 
and $t\in[t_0,T]$.  Moreover, since the solutions $u$ and $v$ can be
analytically extended to $t>0$ in view of the analyticity of the 
multinomial Mittag-Leffler function (e.g., \cite{LY}), 
we have $u(x_0,t)=v(x_0,t)$ for $t>0$. Consequently by the Laplace transform,
we obtain
$$
\sum_{j=1}^{\infty} \rho_j
  \frac{\sum_{i=1}^n q_i \eta^{\alpha_i-1}}
    {\sum_{i=1}^n q_i \eta^{\alpha_i}+\lambda_j}
= \sum_{j=1}^{\infty} \rho_j
  \frac{\sum_{i=1}^m r_i \eta^{\beta_i-1}}
    {\sum_{i=1}^m r_i \eta^{\beta_i}+\lambda_j},\quad \eta>0,
$$
where $\rho_j=(a,\phi_j)\phi_j(x_0)$. 
Moreover, noting $\gamma>\frac{d}{2}-1$, similarly to (\ref{unif-u}), 
we have $\sum_{j=1}^{\infty} |\rho_j|<\infty$. 
Therefore
\begin{align}\label{series Lu=Lv}
\sum_{j=1}^{\infty}\frac{\lambda_j\rho_j}{\sum_{i=1}^n q_i \eta^{\alpha_i}
+\lambda_j}
=\sum_{j=1}^{\infty}\frac{\lambda_j\rho_j}{\sum_{i=1}^m r_i \eta^{\beta_i}
+\lambda_j},
\quad \mbox{$\eta\in\mathbb{R}$ with $|\eta|$ small enough},
\end{align}
where the series on both sides are uniformly convergent for $|\eta|$ small
enough. On the other hand, we set
$$
p_k=(-1)^k\sum_{j=1}^{\infty}\frac{\rho_j}{\lambda_j^k}.
$$
Then 
$$
0 < (-1)^kp_k <\infty, \quad k\in\mathbb{N}.
$$
In fact, since $\sum_{j=1}^{\infty}|\rho_j|<\infty$, and $\lambda_j>0$,
$\lim{\lambda_j}=\infty$, 
we see that $p_k<\infty$. By the assumption of $a$, we have 
$p_1=-\sum_{j=1}^{\infty}\lambda_j^{-1}(a,\phi_j)\phi_j(x_0)=-(A^{-1}a)(x_0)$. 
Setting $b=-A^{-1}a$, we have $Ab=-a$ and $b|_{\partial\Omega}=0$. 
By the strong maximum principle for 
$Au=-\sum_{i,j=1}^d \partial_j(a_{ij}\partial_iu)-cu$ with $c\leq0$ and $a\geq0$, 
we have $b<0$ in $\Omega$. Hence $p_1<0$. 
Similarly, we can prove $(-1)^kp_k>0$ for $k=2,3,\cdots$.

We consider the asymptotic expansion of (\ref{series Lu=Lv}) near $\eta=0$. 
Since $\lambda_j>0$ for $j\in\mathbb{N}$, 
we have
$\left| \frac{\sum_{i=1}^n q_i \eta^{\alpha_i}}{\lambda_j} \right|<1$, 
$\left| \frac{\sum_{i=1}^m r_i \eta^{\beta_i}}{\lambda_j} \right|<1$ for 
small $\eta$ and all $j\in\mathbb{N}$. Consequently
\begin{align}\label{expan-alpha-beta}
\sum_{k=1}^{\infty} p_k \left(\sum_{i=1}^n q_i \eta^{\alpha_i}\right)^k
=\sum_{k=1}^{\infty} p_k \left(\sum_{i=1}^m r_i \eta^{\beta_i}\right)^k,
\mbox{uniformly converges for small $|\eta|$.}
\end{align}

Firstly, we prove $m=n$. Otherwise, we can assume $m>n$. 
Now we proceed by induction to prove that 
$\alpha_i=\beta_i$, $q_i=r_i$, $i=1,\cdots,n$. 
First we prove $\alpha_1=\beta_1$, $q_1=r_1$.
From (\ref{expan-alpha-beta}), we see that
\begin{align*}
 p_1 q_1\eta^{\alpha_1}+p_1\sum_{i=2}^n q_i \eta^{\alpha_i}
+ \sum_{k=2}^{\infty} p_k \left(\sum_{i=1}^n q_i \eta^{\alpha_i}\right)^k
=p_1 r_1\eta^{\beta_1} +p_1\sum_{i=2}^m r_i \eta^{\beta_i}
+\sum_{k=2}^{\infty} p_k \left(\sum_{i=1}^m r_i \eta^{\beta_i}\right)^k.
\end{align*} 
We see that $\alpha_1=\beta_1$ from $p_1<0$, $q_1>0$ and $r_1>0$. 
If not, we can assume that $\alpha_1>\beta_1$.  
Dividing both sides of the above equality by $\eta^{\beta_1}$, 
we obtain
\begin{align}\label{expan-alpha-beta'}
&p_1 q_1\eta^{\alpha_1-\beta_1}+p_1\sum_{i=2}^n q_i \eta^{\alpha_i-\beta_1}
+ \sum_{k=2}^{\infty} p_k \left(\sum_{i=1}^n q_i \eta^{\alpha_i}\right)^k 
\eta^{-\beta_1}                    \nonumber\\
=&p_1 r_1 +p_1\sum_{i=2}^m r_i \eta^{\beta_i-\beta_1}
+\sum_{k=2}^{\infty} p_k \left(\sum_{i=1}^m r_i \eta^{\beta_i}\right)^k 
\eta^{-\beta_1}.
\end{align} 
Now letting $\eta\rightarrow0$, from $\alpha_1>\beta_1$, (\ref{condi-alpha}) 
and (\ref{condi-beta}), we derive that the left-hand side of 
(\ref{expan-alpha-beta'}) tends to $0$, but the right-hand side tends to
$p_1 r_1\neq0$, which is a contradiction.  Hence $\alpha_1\leq\beta_1$. 
By a similar argument, we have $\alpha_1\geq\beta_1$. 
Therefore $\alpha_1=\beta_1$ and $q_1=r_1$.

Suppose for $j\in\mathbb{N}$, $1\leq j\leq n-1$ that $\alpha_i=\beta_i$,
$q_i=r_i$, for $i=1,\cdots,j$, that is
\begin{align}\label{equ-u-v}
\sum_{k=1}^{\infty}  p_k \left(\sum_{i=1}^j q_i \eta^{\alpha_i} 
+ \sum_{i=j+1}^n q_i\eta^{\alpha_i}\right)^k
=\sum_{j=1}^{\infty} p_k \left(\sum_{i=1}^j q_i \eta^{\alpha_i} 
+ \sum_{i=j+1}^m r_i\eta^{\beta_i}\right)^k,
\end{align}
uniformly converges for small $|\eta|$. 
We show that (\ref{equ-u-v}) holds also for $j+1$.

By $S_1$ and $S_2$ we denote the sets of the orders $\ell$ of the 
terms of $\eta^{\ell}$ of each side of (\ref{equ-u-v}) respectively.
For the case 
\begin{align}\label{condi-alpha>beta-notin}
\alpha_{j+1} > \beta_{j+1} \quad \mbox{and} \quad
\beta_{j+1}\notin\left\{\sum_{i=1}^jk_i\alpha_i;\quad k_i\in\mathbb{N}\right\},
\end{align}
from (\ref{condi-alpha}) and (\ref{condi-beta}), it follows that
$$
\beta_{j+1}\notin \left\{\sum_{i=1}^jk_i {\alpha_i} + \sum_{i=j+1}^n k_i 
\alpha_i;\ \ k_i\in\mathbb{N}\right\}.
$$
In fact, if not, then there exist $k_i^0\in\mathbb{N}$ for $i=1,\cdots,n$ 
such that 
$$
\beta_{j+1} = \sum_{i=1}^j k_i^0 {\alpha_i} 
+ \sum_{i=j+1}^n k_i^0 \alpha_i.
$$
Then (\ref{condi-alpha}), (\ref{condi-beta}) and (\ref{condi-alpha>beta-notin}) 
show that $\beta_{j+1}<\alpha_{j+1}<\cdots<\alpha_n$. 
Hence $k_i^0=0$ for $i=j+1,\cdots,n$. 
This means 
$
{\beta}_{j+1}\in\left\{\sum_{i=1}^jk_i\alpha_i;\quad k_i\in\mathbb{N}\right\}
$, 
which is a contradiction. Moreover we can find $\beta_{j+1}\notin S_1$ in view of
$$
S_1\subset\left\{\sum_{i=1}^j k_i {\alpha_i} + \sum_{i=j+1}^n k_i\alpha_i;
 \quad k_i\in\mathbb{N}\right\},
$$
which is a contraction since (\ref{condi-alpha>beta-notin}) yields that 
$\beta_{j+1}\in S_2$. Indeed the coefficient of $\eta^{\beta_{j+1}}$ on the 
right-hand side of (\ref{equ-u-v}) is $p_1r_{j+1}\neq0$.

For the case
\begin{align}\label{condi-alpha>beta-in}
\alpha_{j+1} > \beta_{j+1} \quad \mbox{and}
\quad \beta_{j+1}\in\left\{\sum_{i=1}^j k_i\alpha_i;\quad k_i\in\mathbb{N}\right\},
\end{align}
we now proceed to show that the coefficients of $\eta^{\beta_{j+1}}$ 
on both sides of (\ref{equ-u-v}) are different. Indeed, again using 
assumptions (\ref{condi-alpha}) and (\ref{condi-beta}), 
we find that the coefficient of $\eta^{\beta_{j+1}}$ on the left-hand side of 
(\ref{equ-u-v}) is composed only of the coefficients of $\eta^{\alpha_i}$, 
$i=1,\cdots,j$, that is
$$
\sum_{k_1\alpha_1\cdots+k_j\alpha_j
= \beta_{j+1}} p_{k_1+\cdots+k_j} q_1^{k_1} \cdots q_j^{k_j}.
$$ 
Similarly, we see that the coefficient of $\eta^{\beta_{j+1}}$ on the 
right-hand side of (\ref{equ-u-v}) is
$$
p_1r_{j+1}
+\sum_{k_1\alpha_1\cdots+k_j\alpha_j
= \beta_{j+1}} p_{k_1+\cdots+k_j} 
q_1^{k_1} \cdots q_j^{k_j}.
$$
This is a contradiction since $p_1<0$ and $r_{j+1}>0$. Consequently, 
$\alpha_{j+1}\leq\beta_{j+1}$. 
In the same manner, we can see $\alpha_{j+1}\geq\beta_{j+1}$. Therefore
$$
\sum_{k=1}^{\infty}   p_k \left(\sum_{i=1}^{j+1} q_i \eta^{\alpha_i} 
+ \sum_{i=j+2}^n q_i\eta^{\alpha_i}\right)^k
= \sum_{k=1}^{\infty} p_k \left(\sum_{i=1}^{j+1} q_i \eta^{\alpha_i} 
+ \sum_{i=j+2}^m r_i\eta^{\beta_i}\right)^k.
$$
By induction, we can derive $\alpha_i=\beta_i$ and $q_i=r_i$ for $i=1,\cdots,n$, that is
\begin{equation}\label{expan-u-v-n}
\sum_{k=1}^{\infty}   p_k \left(\sum_{i=1}^n q_i \eta^{\alpha_i}\right)^k
= \sum_{k=1}^{\infty} p_k \left(\sum_{i=1}^n q_i \eta^{\alpha_i} 
+ \sum_{i=n+1}^m r_i\eta^{\beta_i}\right)^k.
\end{equation}
Consequently 
$$
\beta_{n+1}\in\left\{ \sum_{i=1}^n k_i\alpha_i;\quad k_i\in\mathbb{N}\right\}.
$$
This is impossible. In fact, we find that the coefficient of $\eta^{\beta_{n+1}}$ 
on the left-hand side of (\ref{expan-u-v-n}) is 
$$
\sum_{k_1\alpha_1\cdots+k_n\alpha_n
= \beta_{n+1}} p_{k_1+\cdots+k_n} q_1^{k_1} \cdots q_n^{k_n}.
$$ 
and the coefficient of $\eta^{\beta_{n+1}}$ on the right-hand side of 
(\ref{expan-u-v-n}) is
$$
p_1r_{n+1}
+\sum_{k_1\alpha_1\cdots+k_n\alpha_n
= \beta_{n+1}} p_{k_1+\cdots+k_n} 
q_1^{k_1} \cdots q_n^{k_n},
$$
which is a contradiction in view of $r_{n+1}>0$. Therefore, we see that $m > n$ is impossible.
Hence $m\leq n$. Similarly, we can prove that $m\geq n$. Finally we obtain $m=n$ and 
repeat the above argument to obtain $\alpha_i=\beta_i$, $q_i=r_i$, $i=1,\cdots,n$.
\end{proof}


\section{Conclusions and remarks}
 
Theorem \ref{thm-uniqu-delta} establishes the uniqueness in simultaneously 
identifying a number of fractional derivatives as well as fractional orders 
in one-dimensional fractional diffusion equation with initial value given by 
the Dirac delta function by measured data at one endpoint. 
Theorem \ref{thm-uniqu-L2} proves the uniqueness in determining 
fractional orders in the $d$--dimensional diffusion equation with 
$L^2(\Omega)$-initial function by using interior measurement.
A more important inverse problem is simultaneous determination of diffusion 
coefficients as well as fractional orders and in a forthcoming paper
Li, Imanuvilov and Yamamoto \cite{LIY} we discuss this type of inverse 
problem.



\begin{thebibliography}{00}

\bibitem{BY} 
S. Beckers and M. Yamamoto, Regularity and unique existence of solution 
to linear diffusion equation with multiple time-fractional derivatives, 
in: K. Bredies, C. Clason, K. Kunisch, G. von Winckel (Eds.), Control and 
Optimization with PDE Constraints, Birkh\"auser, Basel, 2013, pp. 45--56.
   
\bibitem{CNYY}
M. Cheng, J. Nakagawa, M. Yamamoto and T. Yamazaki, Uniqueness in an 
inverse problem for a one
dimensional fractional diffusion equation. Inverse Problems 25 (2009) 115002.

\bibitem{CH}
R. Courant and D. Hilbert, {\em Methods of Mathematical Physics}, 
Interscience Publishers, New York, 1953.

\bibitem{HNWY}
Y. Hatano, J. Nakagawa, S. Wang and M. Yamamoto, 
Determination of order in fractional diffusion equation. J. Math-for-Ind. 5A 
(2013), 51-57.

\bibitem{LIY}
Z. Li, O. Imanuvilov and M. Yamamoto,
Uniqueness in inverse boundary value problems for fractional diffusion 
equations, preprint.
   
\bibitem{LLY}
Z. Li, Y. Liu and M. Yamamoto, Initial-boundary value problems for multi-term 
time-fractional diffusion equations with positive constant coefficients,
arXiv: 1312.2112, 2013.

\bibitem{LY}
Z. Li and M. Yamamoto, Initial-boundary value problems for linear 
diffusion equation with multiple time-fractional derivatives, 
arXiv:1306.2778v2, 2013.
  
\bibitem{Lu09}
Y. Luchko, Boundary value problems for the generalized time-fractional 
diffusion equation of distributed order, Fract. Calc. Appl. Anal. 12 (2009) 
409-422. 

\bibitem{Luchko2}
Y. Luchko, Initial-boundary-value problems for the generalized multi-term 
time-fractional diffusion equation. J. Math. Anal. Appl. 374 (2011), 538-548.

\bibitem{LG99}
Y. Luchko, R. Gorenflo, An operational method for solving fractional 
differential equations with the Caputo derivatives, Acta Math. Vietnam 24 
(1999) 207--233.

\bibitem{MK}
R. Metzler and J. Klafter, The random walk's guide to anomalous diffusion: 
a fractional dynamics approach, Physics Reports 339 (2000), 1-77.

\bibitem{Podlubny}
I. Podlubny, {\em Fractional Differential Equations}. Academic Press, 
San Diego,1999.
  
\bibitem{Ta}
H. Tanabe, \textit{Equations of Evolution}, Pitman, London, 1979.
\end{thebibliography}
\end{document}